\documentclass[leqno,12pt]{article} 
\usepackage{amsmath, amssymb}
\usepackage{amsthm} 
\theoremstyle{plain} 
\newtheorem{theorem}{\indent\sc Theorem}[section]

\newtheorem{corollary}[theorem]{\indent\sc Corollary}
\newtheorem{proposition}[theorem]{\indent\sc Proposition}

\theoremstyle{definition} 
\newtheorem{definition}[theorem]{\indent\sc Definition}
\newtheorem{remark}[theorem]{\indent\sc Remark}
\newtheorem{example}[theorem]{\indent\sc Example}

\title{Almost $\eta$-Ricci solitons in $(LCS)_n$-manifolds}

\author{Adara M. Blaga}
\date{}
\begin{document}

\maketitle


\begin{abstract}
We consider almost $\eta$-Ricci solitons in $(LCS)_n$-manifolds satisfying certain curvature conditions. We provide a lower and an upper bound for the norm of the Ricci curvature in the gradient case, derive a Bochner-type formula for an almost $\eta$-Ricci soliton and state some consequences of it on an $(LCS)_n$-manifold.
\end{abstract}

\section{Introduction}

Ricci solitons and $\eta$-Ricci solitons (which include quasi-Einstein metrics) are natural generalizations of Einstein metrics. In 1982, R. S. Hamilton introduced on a Riemannian manifold $(M,g_0)$ an evolution equation for metrics, called the \textit{Ricci flow} \cite{hamil}:
$$\frac{\partial}{\partial t}g(t)=-2S(g(t)), \ \ g(0)=g_0,$$
(for $S$ the Ricci curvature tensor), which is used to deform a metric by smoothing out its singularities. Ricci solitons may be regarded as generalized fixed points of the Ricci flow (i.e. fixed points in the quotient space of Riemannian metrics modulo diffeomorphisms and rescalings), modeling the formation of singularities. Precisely, a \textit{Ricci soliton} on a Riemannian manifold $(M,g)$ is defined \cite{hamil} as a triple $(g,\xi,\lambda)$, for $\xi$ a vector field on $M$ and $\lambda$ a real constant, satisfying the equation:
$$\frac{1}{2}\mathcal{L}_{\xi}g+S+\lambda g=0,$$
where $\mathcal{L}_{\xi}$ is the Lie derivative operator along the vector field $\xi$. Obviously, if $\xi$ is a Killing vector field, then the Ricci soliton reduces to an Einstein metric $(g,\lambda)$. A further generalization is the notion of \textit{$\eta$-Ricci soliton} defined by J. T. Cho and M. Kimura \cite{ch} as a quadruple $(g,\xi,\lambda,\mu)$, for $\xi$ a vector field on the Riemannian manifold $(M,g)$ and $\lambda$ and $\mu$ real constants, satisfying the equation:
$$\frac{1}{2}\mathcal{L}_{\xi}g+S+\lambda g+\mu \eta\otimes \eta=0,$$
where $\eta$ is the $g$-dual $1$-form of $\xi$. Moreover, if $\lambda$ and $\mu$ are let to be smooth functions, we talk about \textit{almost Ricci solitons} \cite{pi} and \textit{almost $\eta$-Ricci solitons}, respectively, which we shall consider in our paper.

In the last years, the interest in studying Ricci solitons and their generalizations in different geometrical contexts has considerably increased, due to their connection to general relativity. They control the Ricci and the scalar curvature of the manifold, thus their study towards a special view to different curvature conditions is appropriate.

In the present paper we consider almost $\eta$-Ricci solitons on $(LCS)_n$-manifolds which satisfy certain curvature properties, in particular, $(\xi,\cdot)_{R}\cdot S=0$ and $(\xi,\cdot)_{S}\cdot R=0$, respectively. The \textit{$(LCS)_n$-manifolds} were introduced by A. A. Shaikh \cite{ali} and they are important in the general theory of relativity and cosmology. Different aspects of $\eta$-Ricci solitons in $(LCS)_n$-manifolds have lately been studied by many authors (see \cite{as}, \cite{bai}, \cite{hu}, \cite{huic}). Also, remark that results on Ricci solitons satisfying similar types of curvature conditions have been obtained: in \cite{na} by H. G. Nagaraja and C. R. Premalatha dealed with the cases $(\xi,\cdot)_{R}\cdot \tilde{C}=0$, $(\xi,\cdot)_{P}\cdot \tilde{C}=0$, $(\xi,\cdot)_{H}\cdot S=0$, $(\xi,\cdot)_{\tilde{C}}\cdot S=0$ and in \cite{ba}, C. S. Bagewadi, G. Ingalahalli and S. R. Ashoka treated the cases: $(\xi,\cdot)_{R}\cdot B=0$, $(\xi,\cdot)_{B}\cdot S=0$, $(\xi,\cdot)_{S}\cdot R=0$, $(\xi,\cdot)_{R}\cdot \bar{P}=0$ and $(\xi,\cdot)_{\bar{P}}\cdot S=0$.

\textit{Gradient solitons}, a particular case of solitons having the potential vector field $\xi$ of gradient type, are of special interest, the gradient vector fields playing a central r\^{o}le in different mathematical-physics theories, for example, in Morse-Smale theory \cite{ms}. Our interest is also to characterize the geometry of an almost $\eta$-Ricci soliton in the case when the potential vector field is of gradient type, underlying these results for $(LCS)_n$-manifolds. We provide a lower and an upper bound for the Ricci curvature tensor's norm and derive a Bochner-type formula in the gradient case.

\section{$(LCS)_n$-manifolds}

Let $(M,g)$ be an $n$-dimensional Lorentzian manifold and $\xi$ a unit timelike concircular vector field, that is a vector field satisfying $g(\xi,\xi)=-1$ and $\nabla \xi=\alpha (I+\eta\otimes \xi)$, with $\alpha$ a nowhere zero smooth function on $M$ verifying $d\alpha=\rho \eta$, for $\rho\in C^{\infty}(M)$, where $\nabla$ is the Levi-Civita connection of $g$ and $\eta:=i_{\xi}g$. Denote by $\varphi$ the $(1,1)$-tensor field $\varphi:=\frac{1}{\alpha}\nabla\xi$.

In \cite{ma}, K. Matsumoto introduced the notion of \textit{Lorentzian para-Sasakian structure (briefly, LP-Sasakian structure)}. A more general notion is that of \textit{Lorentzian concircular structure (briefly, $(LCS)_n$-structure)} introduced by A. A. Shaikh \cite{sh} as being the data $(g,\xi,\eta,\varphi,\alpha)$, where $g$ is a Lorentzian metric on $M$, $\xi$ is a unit timelike concircular vector field, $\eta:=i_{\xi}g$ is the $g$-dual of $\xi$ and $\varphi:=I+\eta\otimes \xi$ is the associated $(1,1)$-tensor field.

From the definition it follows that:
\begin{enumerate}
  \item $\varphi \xi =0$, \ $\eta \circ \varphi =0$,
  \item $\eta (\xi )=-1$, \ ${\varphi }^2=I+\eta \otimes \xi $,
  \item $g(\varphi X, Y)=g(X,\varphi Y)$, for any $X$, $Y\in \mathfrak{X}(M)$ and $g(\varphi \cdot, \varphi \cdot)=g+\eta \otimes \eta $,
  \item $(\nabla_X\varphi)Y=\alpha[g(X,Y)\xi+2\eta(X)\eta(Y)\xi+\eta(Y)X]$, for any $X$, $Y\in \mathfrak{X}(M)$.
\end{enumerate}

\medskip

Properties of this structure which follow from straightforward computations are given in the next proposition.

\begin{proposition}\label{p1}
On an $(LCS)_n$-manifold $(M,g,\xi,\eta,\varphi,\alpha)$, for any $X$, $Y$, $Z\in \mathfrak{X}(M)$, the following relations hold:
\begin{equation}\eta(\nabla_X\xi)=0, \ \ \nabla_{\xi}\xi=0,\end{equation}
\begin{equation}R(X,Y)\xi=(\alpha^2-\rho)[\eta(Y)X-\eta(X)Y],\end{equation}
\begin{equation}\eta(R(X,Y)Z)=(\alpha^2-\rho)[\eta(X)g(Y,Z)-\eta(Y)g(X,Z)], \ \ \eta(R(X,Y)\xi)=0,\end{equation}
\begin{equation}\nabla \eta=\alpha(g+\eta\otimes \eta), \ \ \nabla_{\xi} \eta=0,\end{equation}
\begin{equation}\mathcal{L}_{\xi}\varphi=0, \ \ \mathcal{L}_{\xi}\eta=0, \ \ \mathcal{L}_{\xi}g=2 \nabla \eta,\end{equation}
where $R$ is the Riemann curvature tensor field, $\nabla$ is the Levi-Civita connection associated to $g$ and $\mathcal{L}_{\xi}$ denotes the Lie derivative operator along the vector field $\xi$.
\end{proposition}

\begin{remark}
On an $(LCS)_n$-manifold $(M,g,\xi,\eta,\varphi,\alpha)$ we deduce that:

(i) the $1$-form $\eta$ is closed;

(ii) the Nijenhuis tensor field of $\varphi$ vanishes identically, therefore, the structure is normal;

(iii) if $\alpha$ is a constant function and $(M,g)$ is of constant curvature $k$, then $M$ is elliptic manifold and $k=\alpha^2$.
\end{remark}

\bigskip

\begin{example}\cite{hui}\label{e2}
Let $M=\{(x,y,z)\in \mathbb{R}^3, z\neq 0\}$, where $(x,y,z)$ are the standard coordinates in $\mathbb{R}^3$. Consider the linearly independent system of vector fields
$$E_1:=z^2\frac{\partial}{\partial x}, \ \ E_2:=z^2\frac{\partial}{\partial y}, \ \ E_3:=\frac{\partial}{\partial z}.$$

Define the Lorentzian metric $g$ by:
$$g(E_1,E_1)=g(E_2,E_2)=-g(E_3,E_3)=1,$$
$$g(E_1,E_2)=g(E_2,E_3)=g(E_3,E_1)=0$$
the vector field $\xi$ and the $1$-form $\eta$ by:
$$\xi:=E_3, \ \ \eta(X):=g(X,E_3),$$
for any $X\in \mathfrak{X}(M)$, and the $(1,1)$-tensor field $\varphi$ by:
$$\varphi E_1=E_1, \ \ \varphi E_2=E_2, \ \ \varphi E_3=0.$$

Using Koszul's formula for the Lorentzian metric $g$ we obtain:
$$\nabla_{E_1}E_1=-\frac{2}{z}E_3, \ \ \nabla_{E_1}E_2=0, \ \ \nabla_{E_1}E_3=-\frac{2}{z}E_1, \ \ \nabla_{E_2}E_1=0, \ \ \nabla_{E_2}E_2=-\frac{2}{z}E_3,$$$$\nabla_{E_2}E_3=-\frac{2}{z}E_2, \ \ \nabla_{E_3}E_1=0, \ \ \nabla_{E_3}E_2=0, \ \ \nabla_{E_3}E_3=0.$$

In this case, $(g,\xi,\eta,\varphi,\alpha)$ is an $(LCS)_3$-structure on $M$, where $\alpha=-\frac{2}{z}$.
\end{example}

\section{Almost $\eta$-Ricci solitons in $(M,g,\xi,\eta,\varphi,\alpha)$}

A more general notion than almost Ricci soliton \cite{pi} and $\eta$-Ricci soliton \cite{ch}, including also the generalized
quasi-Einstein manifolds \cite{bo}, will be further considered.

Let $(M,g)$ be an $n$-dimensional pseudo-Riemannian manifold ($n>2$), $\xi$ a vector field and $\eta$ a $1$-form on $M$.
\begin{definition}
\textit{An almost $\eta$-Ricci soliton} on $M$ is a data $(g,\xi,\lambda,\mu)$ which satisfy the equation:
\begin{equation}\label{e8}
\mathcal{L}_{\xi}g+2S+2\lambda g+2\mu\eta\otimes \eta=0,
\end{equation}
where $\mathcal{L}_{\xi}$ is the Lie derivative operator along the vector field $\xi$, $S$ is the Ricci curvature tensor field of the metric $g$, and $\lambda$ and $\mu$ are smooth functions on $M$.
\end{definition}

An almost $\eta$-Ricci soliton $(g,\xi,\lambda,\mu)$ is said to be \textit{steady} if $\lambda=0$, \textit{shrinking} if $\lambda<0$ or \textit{expanding} if $\lambda>0$.

In the same way we define \textit{the almost $\eta$-Einstein soliton} as a data $(g,\xi,\lambda,\mu)$ which satisfy the equation:
\begin{equation}
\mathcal{L}_{\xi}g+2S+(2\lambda-scal) g+2\mu\eta\otimes \eta=0,
\end{equation}
where $scal$ is the scalar curvature of $(M,g)$.

\bigskip

Replacing $\mathcal{L}_{\xi}g$ in terms of the Levi-Civita connection $\nabla$ in (\ref{e8}), we obtain:
\begin{equation}\label{e9}
2S(X,Y)=
-g(\nabla_X\xi,Y)-g(X,\nabla_Y\xi)-2\lambda g(X,Y)-2\mu\eta(X)\eta(Y),
\end{equation}
for any $X$, $Y\in \mathfrak{X}(M)$.

\bigskip

If $(M,g,\xi,\eta,\varphi,\alpha)$ is an $(LCS)_n$-manifold, then (\ref{e9}) becomes:
\begin{equation}\label{ea9}
S(X,Y)=
-(\alpha+\lambda) g(X,Y)-(\alpha+\mu)\eta(X)\eta(Y),
\end{equation}
for any $X$, $Y\in \mathfrak{X}(M)$, therefore $M$ is quasi-Einstein manifold.

\bigskip

Also remark that on an $(LCS)_n$-manifold $(M,g,\xi,\eta,\varphi,\alpha)$, the Ricci curvature tensor field satisfies:
\begin{equation}\label{ad1}
S(X,\xi)=(n-1)(\alpha^2-\rho)\eta(X),
\end{equation}
\begin{equation}
S(\varphi X,\varphi Y)=S(X,Y)+(n-1)(\alpha^2-\rho)\eta(X)\eta(Y),
\end{equation}
for any $X$, $Y\in \mathfrak{X}(M)$. From (\ref{ea9}) and (\ref{ad1}) we obtain:
\begin{equation}\label{ad}
\mu-\lambda=(n-1)(\alpha^2-\rho)
\end{equation}
and we can state:

\begin{proposition}
The scalar curvature of an $(LCS)_n$-manifold $(M,g,\xi,\eta,\varphi,\alpha)$ admitting an almost $\eta$-Ricci soliton $(g,\xi,\lambda,\mu)$ is:
\begin{equation}
scal=(1-n)[\alpha-n(\alpha^2+\xi(\alpha))+\mu].
\end{equation}
\end{proposition}

In particular, $M$ is of constant scalar curvature if and only if $d\mu=(1-2n\alpha)\xi(\alpha)\eta+n d(\xi(\alpha))$.

\bigskip

\begin{example}
On the $(LCS)_3$-manifold $(M, g,\xi,\eta,\varphi,\alpha)$ considered in Example \ref{e2}, the data $(g,\xi,\lambda,\mu)$ for $\lambda=\frac{2(z-5)}{z^2}$ and $\mu=\frac{2(z+1)}{z^2}$ defines an almost $\eta$-Ricci soliton.
Indeed, the Riemann and the Ricci curvature tensor fields are given by:
$$R(E_1,E_2)E_2=\frac{4}{z^2}E_1, \ \ R(E_1,E_3)E_3=-\frac{6}{z^2}E_1, \ \ R(E_2,E_1)E_1=\frac{4}{z^2}E_2,$$
$$R(E_2,E_3)E_3=-\frac{6}{z^2}E_2, \ \ R(E_3,E_1)E_1=\frac{6}{z^2}E_3, \ \ R(E_3,E_2)E_2=\frac{6}{z^2}E_3,$$
$$S(E_1,E_1)=S(E_2,E_2)=\frac{10}{z^2}, \ \ S(E_3,E_3)=-\frac{12}{z^2}.$$
Also, $\alpha=-\frac{2}{z}$, $\rho=-\frac{2}{z^2}$ and from (\ref{ea9}) we obtain $S(E_1,E_1)=-(\alpha+\lambda)$ and $S(E_3,E_3)=\lambda-\mu$, therefore $\lambda=\frac{2(z-5)}{z^2}$ and $\mu=\frac{2(z+1)}{z^2}$.
\end{example}

\bigskip

Like for the case of $\eta$-Ricci solitons on Lorentzian para-Sasakian manifolds \cite{bl}, the next theorems formulate results in the more general case of $(LCS)_n$-manifold when it is Ricci symmetric, has Codazzi or cyclic $\eta$-recurrent Ricci curvature tensor.

\begin{proposition}\label{t1}
Let $(g,\xi,\eta,\varphi,\alpha)$ be an $(LCS)_n$-structure on the manifold $M$ and let $(g,\xi,\lambda,\mu)$ be an almost $\eta$-Ricci soliton on $M$.
\begin{enumerate}
  \item If the manifold $(M,g)$ is Ricci symmetric (i.e. $\nabla S=0$), then $\alpha^2+\xi(\alpha)$ is locally constant.
  \item If the Ricci tensor is $\eta$-recurrent (i.e. $\nabla S=\eta\otimes S$), then the scalar function $\alpha$ verifies $\alpha^2+(1+2\alpha)\xi(\alpha)+\xi(\xi(\alpha))=0$.
  \item If the Ricci tensor is Codazzi (i.e. $(\nabla_X S)(Y,Z)=(\nabla_Y S)(X,Z)$, for any $X$, $Y$, $Z\in \mathfrak{X}(M)$), then $d(\alpha^2+\xi(\alpha))\otimes \eta=\eta\otimes d(\alpha^2+\xi(\alpha))$.
      \end{enumerate}
\end{proposition}
\begin{proof}
Replacing the expression of $S$ from (\ref{ea9}) in $(\nabla_XS)(Y,Z):=X(S(Y,Z))-S(\nabla_XY,Z)-S(Y,\nabla_XZ)$ we obtain:
     \begin{equation}\label{m}
     (\nabla_XS)(Y,Z)=-(d\alpha+d\lambda)(X)g(Y,Z)-(d\alpha+d\mu)(X)\eta(Y)\eta(Z)-\end{equation}
     $$-\alpha(\alpha+\mu)[g(X,Y)\eta(Z)+g(X,Z)\eta(Y)+2\eta(X)\eta(Y)\eta(Z)].$$
\begin{enumerate}
 \item If $\nabla S=0$, taking $Y:=\xi$ and $Z:=\xi$ in the expression of $\nabla S$ from (\ref{m}) we obtain $d(\lambda-\mu)=0$, therefore $\lambda-\mu$ is locally constant. Also, from (\ref{ad}) we deduce that $\alpha^2-\rho=\alpha^2+\xi(\alpha)$ is locally constant.
      \item If $\nabla S=\eta\otimes S$, taking $Y:=\xi$ and $Z:=\xi$ in (\ref{m}) we obtain $d(\lambda-\mu)=(\lambda-\mu)\eta$, therefore, $d(\alpha^2+\xi(\alpha))=(\alpha^2+\xi(\alpha))\eta$. Replacing $d\alpha=-\xi(\alpha)\eta$ and applying $\xi$ we get the required relation.
\item If  $(\nabla_X S)(Y,Z)=(\nabla_Y S)(X,Z)$, for any $X$, $Y$, $Z\in \mathfrak{X}(M)$, taking $Z:=\xi$ in (\ref{m}) we obtain $d(\lambda-\mu)(X)\eta(Y)=d(\lambda-\mu)(Y)\eta(X)$, for any $X$, $Y\in \mathfrak{X}(M)$ and from (\ref{ad}) we get the conclusion.
      \end{enumerate}
\end{proof}

\bigskip

In what follows we shall consider almost $\eta$-Ricci solitons in $(LCS)_n$-manifolds requiring for the curvature to satisfy $(\xi,\cdot)_{R}\cdot S=0$ and $(\xi,\cdot)_{S}\cdot R=0$, respectively, where by $\cdot$ we denote the derivation of the tensor algebra at each point of the tangent space:
\begin{itemize}
  \item $((\xi,X)_{R}\cdot S)(Y,Z):=((\xi\wedge_RX)\cdot S)(Y,Z):=S((\xi\wedge_RX)Y,Z)+S(Y,(\xi\wedge_RX)Z)$, for $(X\wedge_RY)Z:=R(X,Y)Z$;
  \item $((\xi,X)_{S}\cdot R)(Y,Z)W:=(\xi\wedge_SX)R(Y,Z)W+R((\xi\wedge_SX)Y,Z)W+ R(Y,(\xi\wedge_SX)Z)W+R(Y,Z)(\xi\wedge_SX)W$, for $(X\wedge_SY)Z:=S(Y,Z)X-S(X,Z)Y$.
\end{itemize}

Remark that properties of $\eta$-Ricci solitons when asking for some curvature conditions were discussed in other papers of the author, in different geometries (see \cite{bl}, \cite{blag}, \cite{blaga}, \cite{blagam}, \cite{blcr}).

\begin{theorem}\label{t3}
Let $(g,\xi,\eta,\varphi,\alpha)$ be an $(LCS)_n$-structure on the manifold $M$ and $(g,\xi,\lambda,\mu)$ an almost $\eta$-Ricci soliton on $M$. If $(\xi,\cdot)_{R}\cdot S=0$, then $\mu=-\alpha$ and $\lambda=-\alpha-(n-1)(\alpha^2+\xi(\alpha))$ or $\lambda=\mu$ and $\alpha^2+\xi(\alpha)=0$.
\end{theorem}
\begin{proof}
The condition that must be satisfied by $S$ is:
\begin{equation}
S(R(\xi,X)Y,Z)+S(Y,R(\xi,X)Z)=0,
\end{equation}
for any $X$, $Y$, $Z\in \mathfrak{X}(M)$.

Replacing the expression of $S$ from (\ref{ea9}) and using the symmetries of $R$ we get:
\begin{equation}
(\alpha^2-\rho)(\alpha+\mu)[\eta(Y)g(X,Z)+\eta(Z)g(X,Y)+2\eta(X)\eta(Y)\eta(Z)]=0,
\end{equation}
for any $X$, $Y$, $Z\in \mathfrak{X}(M)$.

For $Z:=\xi$ we have:
\begin{equation}\label{ej}
(\alpha^2-\rho)(\alpha+\mu)[g(X,Y)+\eta(X)\eta(Y)]=0,
\end{equation}
for any $X$, $Y\in \mathfrak{X}(M)$. It follows $\alpha+\mu=0$ or $\alpha^2-\rho=0$. In the first case, from (\ref{ad}) we get $\lambda=-\alpha-(n-1)(\alpha^2+\xi(\alpha))$ and in the second case, also from (\ref{ad}) we obtain $\lambda=\mu$ and $\alpha^2+\xi(\alpha)=0$.
\end{proof}

For $\mu=0$, from Theorem \ref{t3} we deduce:

\begin{corollary}
On an $(LCS)_n$-manifold $(M,g, \xi,\eta,\varphi,\alpha)$ satisfying $(\xi,\cdot)_{R}\cdot S=0$, the Ricci soliton is steady.
\end{corollary}

From the relations (\ref{ea9}), (\ref{ad}) and (\ref{ej}), if $\alpha^2+\xi(\alpha)\neq 0$, assuming that $(\xi,\cdot)_{R}\cdot S=0$, we obtain:
\begin{equation}
S=-(\alpha+\lambda)g=(n-1)(\alpha^2+\xi(\alpha))g.
\end{equation}

Therefore:
\begin{proposition}\label{p2}
If $(g, \xi,\eta,\varphi,\alpha)$ is an $(LCS)_n$-structure on the manifold $M$, $(g,\xi,\lambda,\mu)$ is an almost $\eta$-Ricci soliton on $M$ and $(\xi,\cdot)_{R}\cdot S=0$, then the scalar curvature of $M$ equals to $n(n-1)(\alpha^2+\xi(\alpha))$.
\end{proposition}

\begin{remark}
If $(g, \xi,\eta,\varphi,\alpha)$ is an $(LCS)_n$-structure on the manifold $M$ with $\alpha$ a nonzero constant function, $(g,\xi,\lambda,\mu)$ is an almost $\eta$-Ricci soliton on $M$ and $(\xi,\cdot)_{R}\cdot S=0$, then the $\eta$-Ricci soliton is steady if $\alpha=-\frac{1}{n-1}$, shrinking if $\alpha\in(-\infty, -\frac{1}{n-1})\cup (0,\infty)$ or expanding if $\alpha\in (-\frac{1}{n-1},0)$, and in this case, the scalar curvature of $M$ is positive.
\end{remark}

\begin{theorem}\label{t4}
Let $(g, \xi,\eta,\varphi,\alpha)$ be an $(LCS)_n$-structure on the manifold $M$ and $(g,\xi,\lambda,\mu)$ an almost $\eta$-Ricci soliton on $M$. If $(\xi,\cdot)_{S}\cdot R=0$, then $\mu=-\alpha+2(n-1)(\alpha^2+\xi(\alpha))$ and $\lambda=-\alpha+(n-1)(\alpha^2+\xi(\alpha))$ or $\alpha^2+\xi(\alpha)=0$.
\end{theorem}
\begin{proof}
The condition that must be satisfied by $S$ is:
$$S(X,R(Y,Z)W)\xi-S(\xi,R(Y,Z)W)X+S(X,Y)R(\xi,Z)W-S(\xi,Y)R(X,Z)W+$$
\begin{equation}\label{e777}+S(X,Z)R(Y,\xi)W-S(\xi,Z)R(Y,X)W
+S(X,W)R(Y,Z)\xi-S(\xi,W)R(Y,Z)X=0,
\end{equation}
for any $X$, $Y$, $Z$, $W\in \mathfrak{X}(M)$.

Taking the inner product with $\xi$, the relation (\ref{e777}) becomes:
$$-S(X,R(Y,Z)W)-S(\xi,R(Y,Z)W)\eta(X)+S(X,Y)\eta(R(\xi,Z)W)-$$$$-S(\xi,Y)\eta(R(X,Z)W)
+S(X,Z)\eta(R(Y,\xi)W)-S(\xi,Z)\eta(R(Y,X)W)+$$\begin{equation}+S(X,W)\eta(R(Y,Z)\xi)-S(\xi,W)\eta(R(Y,Z)X)=0,
\end{equation}
for any $X$, $Y$, $Z$, $W\in \mathfrak{X}(M)$.

Replacing the expression of $S$ from (\ref{ea9}) and computing it in $Z:=\xi$ and $W:=\xi$, we get:
\begin{equation}\label{rt}
(\alpha^2-\rho)(\alpha+2\lambda-\mu)[g(X,Y)+\eta(X)\eta(Y)]=0,
\end{equation}
for any $X$, $Y\in \mathfrak{X}(M)$ and we obtain $\mu=-\alpha+2(n-1)(\alpha^2-\rho)$ and $\lambda=-\alpha+(n-1)(\alpha^2-\rho)$ or $\alpha^2-\rho=0$.
\end{proof}

For $\mu=0$, from Theorem \ref{t4} we deduce:

\begin{corollary}
On an $(LCS)_n$-manifold $(M,g,\xi,\eta,\varphi,\alpha)$ satisfying $(\xi,\cdot)_{S}\cdot R=0$, the almost Ricci soliton is given by $\lambda=-(n-1)(\alpha^2+\xi(\alpha))$.
\end{corollary}

From the relations (\ref{ea9}), (\ref{ad}) and (\ref{rt}), if $\alpha^2+\xi(\alpha)\neq 0$, assuming that $(\xi,\cdot)_{S}\cdot R=0$, we obtain:
\begin{equation}
S=-(\alpha+\lambda)(g+2\eta\otimes \eta)=-(n-1)(\alpha^2+\xi(\alpha))(g+2\eta\otimes \eta).
\end{equation}

Therefore:
\begin{proposition}\label{p3}
If $(g,\xi,\eta,\varphi,\alpha)$ is an $(LCS)_n$-structure on the manifold $M$, $(g,\xi,\lambda,\mu)$ is an almost $\eta$-Ricci soliton on $M$ and $(\xi,\cdot)_{S}\cdot R=0$, then the scalar curvature of $M$ equals to $-(n-1)(n-2)(\alpha^2+\xi(\alpha))$.
\end{proposition}

\begin{remark}
If $(g, \xi,\eta,\varphi,\alpha)$ is an $(LCS)_n$-structure on the manifold $M$ with $\alpha$ a nonzero constant function, $(g,\xi,\lambda,\mu)$ is an almost $\eta$-Ricci soliton on $M$ and $(\xi,\cdot)_{S}\cdot R=0$, then the $\eta$-Ricci soliton is steady if $\alpha=\frac{1}{n-1}$, shrinking if $\alpha\in (0,\frac{1}{n-1})$ or expanding if $\alpha\in(-\infty, 0)\cup (\frac{1}{n-1},\infty)$, and in this case, the scalar curvature of $M$ is non positive.
\end{remark}

\begin{remark}
From Proposition \ref{p2} and Proposition \ref{p3} we notice that for $n>2$, the curvature conditions $(\xi,\cdot)_{R}\cdot S=0$ and $(\xi,\cdot)_{S}\cdot R=0$ respectively, on an $(LCS)_n$-manifold admitting an almost $\eta$-Ricci soliton, determine the scalar curvature to be (in each point) of opposite signs, respectively.
\end{remark}

\section{Gradient almost $\eta$-Ricci solitons}

When the potential vector field of (\ref{e8}) is of gradient type, i.e. $\xi=grad(f)$, then $(g,\xi,\lambda,\mu)$ is said to be a \textit{gradient almost $\eta$-Ricci soliton} and the equation satisfied by it becomes:
\begin{equation}\label{e22}
Hess(f)+S+\lambda g+\mu\eta\otimes \eta=0,
\end{equation}
where $Hess(f)$ is the Hessian of $f$ defined by $Hess(f)(X,Y):=g(\nabla_X\xi,Y)$.

\begin{proposition}
Let $(M,g,\xi,\eta,\varphi,\alpha)$ be an $(LCS)_n$-manifold. If (\ref{e22}) defines a gradient almost $\eta$-Ricci soliton on $M$
with the potential vector field $\xi:=grad(f)$ and $\eta=df$ the $g$-dual of $\xi$, then:
\begin{equation}
(\nabla_XQ)Y-(\nabla_YQ)X=(\alpha^2(df\otimes I-I\otimes df)-(d\alpha \otimes I-I\otimes d\alpha)-\end{equation}$$-[(d\lambda-\alpha \mu df)\otimes I-I\otimes (d\lambda-\alpha \mu df)]-[d(\alpha+\mu)\otimes df-df\otimes d(\alpha+\mu)]\otimes \xi)(X,Y),
$$
for any $X$, $Y\in\mathfrak{X}(M)$, where $Q$ stands for the Ricci operator defined by $g(QX,Y):=S(X,Y)$.
\end{proposition}
\begin{proof}
Notice that (\ref{e22}) can be written:
\begin{equation}\label{e23}
\nabla\xi+Q+\lambda I+\mu df\otimes \xi=0.
\end{equation}

Then:
\begin{equation}
(\nabla_XQ)Y=-(\nabla_X\nabla_Y\xi-\nabla_{\nabla_XY}\xi)-X(\lambda)Y-X(\mu)df(Y)\xi-\mu[g(Y,\nabla_X\xi)\xi+df(Y)\nabla_X\xi].
\end{equation}

Replacing now $\nabla\xi=\alpha(I+df\otimes \xi)$ in the previous relation, after a long but straightforward computation, we get the required relation.
\end{proof}

\begin{remark}
i) Remark that since $\xi$ is concircular, hence geodesic vector field, from (\ref{e23}) follows that $\xi$ is an eigenvector of $Q$ corresponding to the eigenvalue $(n-1)(\alpha^2+\xi(\alpha))$. In particular, if $\lambda=\mu$, then $\xi\in \ker Q$.

ii) The Ricci operator is $\varphi$-invariant (i.e. $Q\circ \varphi=\varphi \circ Q$).
\end{remark}

From the above considerations, we can state:
\begin{proposition}
Let $(M,g,\xi,\eta,\varphi,\alpha)$ be an $(LCS)_n$-manifold. If (\ref{e22}) defines a gradient almost $\eta$-Ricci soliton on $M$
with the potential vector field $\xi:=grad(f)$ and $\eta=df$ the $g$-dual of $\xi$ and $M$ is $\varphi$-Ricci symmetric (i.e. $\varphi^2\circ \nabla Q=0$), then the soliton is given by $\lambda=-\alpha-(n-1)(\alpha^2+\xi(\alpha))$ and $\mu=-\alpha$.
\end{proposition}

\bigskip

A lower and an upper bound of the Ricci curvature tensor's norm for a gradient almost $\eta$-Ricci soliton will be further obtained inspired by the inequality in the gradient Ricci soliton case discussed by M. Crasmareanu in \cite{cr}, but using a slightly different argument.

\begin{theorem}
If (\ref{e22}) defines a gradient almost $\eta$-Ricci soliton on the $n$-dimensional pseudo-Riemannian manifold $(M,g)$ and
$\eta=df$ is the $g$-dual of the gradient vector field $\xi:=grad(f)$, then:
\begin{equation}\label{e21}
|\nabla \xi|^2+\mu^2|\xi|^4+\mu \nabla_{\xi}(|\xi|^2)-\frac{(\Delta(f)+\mu|\xi|^2)^2}{n}\leq |S|^2\leq |\nabla \xi|^2+\mu^2|\xi|^4+\mu \nabla_{\xi}(|\xi|^2)+\frac{(scal)^2}{n}.
\end{equation}
\end{theorem}
\begin{proof}
From (\ref{e22}) we obtain:
\begin{equation}\label{e3}
|Hess(f)|^2=|S|^2+\lambda^2n+2\lambda scal-\mu^2|\xi|^4-\mu\xi(|\xi|^2)
\end{equation}
and
\begin{equation}\label{e4}
|S|^2=|Hess(f)|^2+\lambda^2n+2\lambda (\Delta(f)+ \mu|\xi|^2)+\mu^2|\xi|^4+\mu\xi(|\xi|^2).
\end{equation}

Remark that the conditions to exist a solution (in $\lambda$) are:
\begin{equation}
(scal)^2-n[|S|^2-|Hess(f)|^2-\mu^2|\xi|^4-\mu\xi(|\xi|^2)]\geq 0
\end{equation}
and
\begin{equation}
(\Delta(f)+ \mu|\xi|^2)^2-n[|Hess(f)|^2-|S|^2+\mu^2|\xi|^4+\mu\xi(|\xi|^2)]\geq 0
\end{equation}
which just imply the double inequality from the conclusion.
\end{proof}

\begin{remark}
i) A similar estimation holds for gradient almost $\eta$-Einstein solitons: the lefthand side of (\ref{e21}) is exactly the same, but in the righthand side term of (\ref{e21}), $\mu \cdot scal \cdot |\xi|^2$ will be supplementary added.

ii) The corresponding inequalities are just the same for the particular cases of gradient $\eta$-Ricci and gradient $\eta$-Einstein solitons, and in the case of gradient Ricci soliton, we get the result formulated by M. Crasmareanu in \cite{cr}.

iii) If $\xi$ is of constant length, $|\xi|^2=:k$, then (\ref{e21}) simplifies to:
$$|\nabla \xi|^2+\mu^2k^2-\frac{(\Delta(f)+\mu k)^2}{n}\leq |S|^2\leq |\nabla \xi|^2+\mu^2k^2+\frac{(scal)^2}{n}.$$

iv) The simultaneous equalities hold for $(scal)^2=-(\Delta(f)+\mu |\xi|^2)^2 \ (=0)$ i.e. for steady gradient almost $\eta$-Ricci soliton ($\lambda=0$) with $scal=0$ and $\Delta(f)=-\mu |\xi|^2$. In this case, if $|\xi|^2=:k$ is constant, then $|S|^2=|\nabla \xi|^2+\mu^2 k^2$. Also, for an $(LCS)_n$-manifold
$(M,g,\xi,\eta,\varphi,\alpha)$, the gradient almost $\eta$-Ricci soliton is given by $\lambda=0$ and $\mu=(n-1)\alpha$ and the scalar function $\alpha$ must verify $\alpha^2-\alpha+\xi(\alpha)=0$.
\end{remark}

\bigskip

A Bochner-type formula will be obtained for the gradient almost $\eta$-Ricci soliton case.

\begin{theorem}\label{t}
If (\ref{e22}) defines a gradient almost $\eta$-Ricci soliton on the $n$-dimensional pseudo-Riemannian manifold $(M,g)$ and
$\eta=df$ is the $g$-dual of the gradient vector field $\xi:=grad(f)$, then:
\begin{equation}\label{e53}
\frac{1}{2}(\Delta-\nabla_{\xi})(|\xi|^2)=|\nabla \xi|^2+\lambda |\xi|^2+\mu |\xi|^2(|\xi|^2-2\Delta(f))+(n-2)\xi(\lambda)-|\xi|^2\xi(\mu).
\end{equation}
\end{theorem}
\begin{proof}
First remark that:
$$
trace(\mu \eta\otimes \eta)=\mu|\xi|^2
$$
and
$$div(\mu \eta\otimes \eta)=\frac{\mu}{2}d(|\xi|^2)+\mu \Delta(f) df+d\mu(\xi)df.$$

Taking the trace of the equation (\ref{e22}), we obtain:
\begin{equation}\label{e13}
\Delta(f)+scal+n\lambda +\mu |\xi|^2=0
\end{equation}
and differentiating it:
\begin{equation}\label{e14}
d(\Delta(f))+d(scal)+nd\lambda+\mu d(|\xi|^2)+|\xi|^2d\mu=0.
\end{equation}

Now taking the divergence of the same equation, we get:
\begin{equation}\label{e15}
div(Hess(f))+div(S)+d\lambda+\frac{\mu}{2}d(|\xi|^2)+\mu \Delta(f) df+d\mu(\xi)df=0.
\end{equation}

Substracting the relations (\ref{e15}) and (\ref{e14}) computed in $\xi$ and using \cite{bla}:
$$S(\xi,\xi)=-\frac{1}{2}\xi(|\xi|^2)-\lambda |\xi|^2-\mu|\xi|^4,$$
$$div(Hess(f))=d(\Delta(f))+i_{Q\xi}g,$$
$$(div(Hess(f)))(\xi)=\frac{1}{2}\Delta(|\xi|^2)-|\nabla \xi|^2$$
and
$$
div(S)=\frac{1}{2}d(scal),$$
we obtain (\ref{e53}).
\end{proof}

\begin{remark}
Denoting by $\Delta_f:=\Delta -\nabla_{\xi}$ the diffusion operator, for $\mu=0$ in Theorem \ref{t}, we get the relation for the case of gradient almost Ricci soliton:
\begin{equation}
\frac{1}{2}\Delta_f(|\xi|^2)=|\nabla \xi|^2+\lambda |\xi|^2+(n-2)\xi(\lambda)
\end{equation}
and in particular, for $\lambda$ constant, we obtain the corresponding relation for gradient Ricci soliton \cite{pe}:
\begin{equation}
\frac{1}{2}\Delta_f(|\xi|^2)=|\nabla \xi|^2+\lambda |\xi|^2.
\end{equation}
\end{remark}

\begin{remark}
For the case $\mu=0$, under the assumptions $\lambda |\xi|^2 \geq (2-n)\xi(\lambda)$ we get $\Delta_f(|\xi|^2)\geq 0$ and from the maximum principle follows that $|\xi|^2$ is constant in a neighborhood of any local maximum. If $|\xi|$ achieve its maximum, then $\lambda =\frac{2-n}{|\xi|^2}\xi(\lambda)$ (for $\xi\neq 0$), which yields a steady gradient Ricci soliton if $\lambda$ is constant.
\end{remark}

\begin{remark}
From Theorem \ref{t}, using (\ref{ad}) and taking into account that on an $(LCS)_n$-manifold we have $|\nabla \xi|^2=\alpha^2(n-1)$, we deduce that:

i) if $(g,\xi,\lambda,\mu)$ is a gradient almost $\eta$-Ricci soliton, then:
$$2\alpha \mu+\xi(\mu)=-2\alpha^2+[2(n-2)\alpha-1]\xi(\alpha)+(n-2)\xi(\xi(\alpha));$$

ii) if $(g,\xi,\lambda,\mu)$ is a gradient $\eta$-Ricci soliton, then $\alpha$ is constant, $\lambda=-\alpha-(n-1)\alpha^2$ and $\mu=-\alpha$.

iii) if $(g,\xi,\lambda)$ is a gradient almost Ricci soliton, then the scalar function $\alpha$ must satisfy:
$$2\alpha^2-[2(n-2)\alpha-1]\xi(\alpha)-(n-2)\xi(\xi(\alpha))=0;$$

iv) there is no gradient Ricci soliton on $M$.
\end{remark}

\small{

\bigskip

\textit{Adara M. Blaga}

\textit{Department of Mathematics}

\textit{West University of Timi\c{s}oara}

\textit{Bld. V. P\^{a}rvan nr. 4, 300223, Timi\c{s}oara, Rom\^{a}nia}

\textit{adarablaga@yahoo.com}
}
\end{document}